\documentclass[11pt]{amsart}
\usepackage{amssymb}
\usepackage{graphicx} 
\usepackage{enumerate}
\usepackage{mathtools}
\usepackage{color,soul}
\usepackage{cite}
\usepackage{tikz}
\usetikzlibrary{matrix}
\usepackage{caption}
\usepackage{amsmath, amsthm}
\usepackage[colorlinks=true, allcolors=blue]{hyperref}
\usepackage[noabbrev]{cleveref}
 \usepackage{relsize}
 \usetikzlibrary{cd}
 \usetikzlibrary{decorations.pathreplacing}
 \usepackage{tikz,calc}
\usepackage{color}
\usepackage[margin=1.2in]{geometry}

\usepackage{multicol,todonotes}

\newtheorem{theorem}{Theorem}[section]

\theoremstyle{definition}

\newtheorem*{KeyFact}{Key Fact}
\newtheorem*{Cassini}{Cassini's Identity}

\title{Winning Lights Out with Fibonacci}
\author{Crista Arangala, Stephen Bailey, Kristen Mazur}

\begin{document}

\maketitle

\section{Introduction}

Before the days of Candy Crush and Angry Birds, long car rides for teenagers meant being stuck in the back seat, using the breeze from an open window to stay cool and small handheld games to stay entertained. One such game was the electronic, single-player game Lights Out by Tiger Electronics. In the original Lights Out game, the player is presented with a $5 \times 5$ grid of lights, some of which are on and others off. The objective is to turn off all lights, but pressing a button not only changes its state (on to off or vice versa) but also changes the states of the lights above, below, left, and right. Mathematicians have found   this game intriguing because we can analyze it using mathematics such as linear algebra and graph theory. For examples, see \cite{anderson}, \cite{ferman}, \cite{martin}, \cite{miss}, and \cite{torrence}.

But, what if we are now faced with an $m \times n$ grid of lights of varying brightness (or $k$ states), and to make things more fun we place the grid on a cylinder so that the left buttons  connect to the right. Then, with fatigue from boredom setting in on our long, phone-free car ride, we start playing the game using a strategy that uses minimal brain power; light chasing. We realize that we can turn off a light by pressing the button directly below it. So, when using light chasing we systematically turn off the lights one row at a time by pressing the appropriate lights in the row below. After pressing the last row of lights, it is likely that some lights in that  row will remain on. When this occurs in the original Lights Out game (whose grid of buttons lies in the plane instead of on a cylinder) we can often then press specific lights in the top row, work down the grid again using the light chasing strategy, and on this second pass, all lights will be turned off. While there is no quick mathematical way to determine which buttons to press at the top after the first pass of light chasing,  there are ``Look Up" tables that give us this information. In \cite{leach}, Leach uses linear algebra to explain how to generate these tables for the original game.

Instead of trying to generate ``Look Up" tables for the game whose grid is laid on a cylinder, we realize that when we start the game with all lights set to the same setting and when the number of rows and number of light states are just right, all lights are already turned off when we reach the last row of lights on the first pass. (So long as the board has at least three  columns, the number of columns does not seem to affect the game.) Teased with this revelation, we ponder the main question of this paper. \textit{If the game begins with all lights set to the same state, how many rows of buttons should the cylindrical board  have if the lights have $k$ states and we want to turn off all lights using only the first pass of the light chasing strategy.} From here on, we refer to the strategy of only completing the first pass of the light chasing process as \textit{one-pass chasing}.

The appeal of one-pass chasing on the cylindrical game  is that not only   can we play it on long car rides when our phones die, but also that its mathematical analysis invokes a connection to the Fibonacci sequence and properties thereof.  In the remainder of this paper we build a recursive model for the naive light chasing strategy on the cylindrical Lights Out board and study the recursion under various moduli. The connection to the Fibonacci sequence  allows us to use facts about the Fibonacci numbers (mod $k$) to provide solutions to our main  question.

\section{One-Pass Chasing on a Cylinder}\label{sec:basics}

Following the familiar construction of a cylinder by identifying the left and right sides of a rectangle, we visualize the cylindrical Lights Out board as a grid of light-up buttons in which pressing the leftmost button in a row affects the rightmost button (and vice versa).  Further, we represent the $k$ states of the lights as the $k$ congruence classes (mod $k$), with 0 representing the ``off" state. Pressing a button one times increases the state of its light by one (mod $k$). In particular, pressing a button whose light is at the $k-1$ state one time reduces the light to the 0 state, i.e., it turns the light off. In \Cref{fig:BasicEx} we show a cylindrical Lights Out game in which the lights have three states. Pressing the starred button on the board on the left results in the board on the right.

\begin{figure}[h]
\includegraphics[width=3.5in]{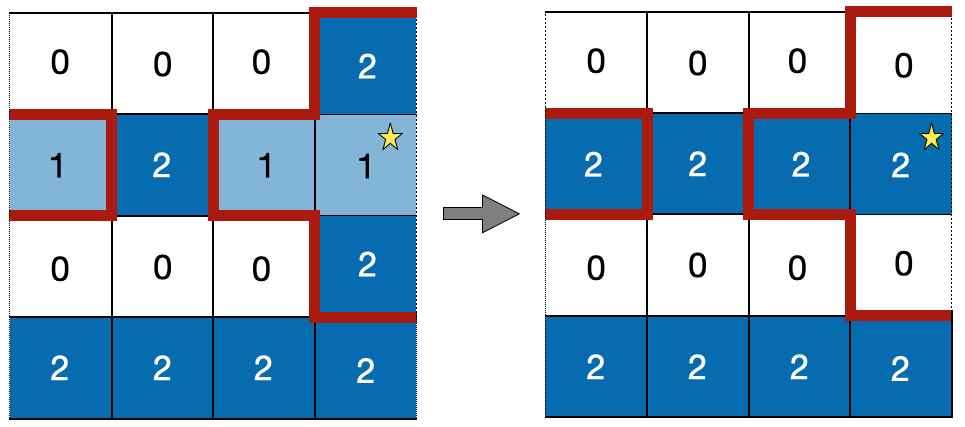}
\caption{Lights Out played on a cylindrical grid. The lights have three states. Pressing the starred button on the left results in the board on the right.}
\label{fig:BasicEx}
\end{figure}

We begin the one-pass chasing strategy by pressing the buttons in the second row to turn off the lights in the first row. Since we are playing on a cylinder, it does not matter where in the second row we start, but by convention we will start with the column drawn on the leftmost side. Once all lights in the first row are turned off, we proceed to turn off the lights in the second row by pressing the lights in the third. We continue with this method until we press the buttons in the last row. 

In \Cref{fig:OnePassEx} we show the one-pass chasing process on a cylindrical game in which the board has five rows, the lights have four states, and all lights start at the 3 state. The number in each star is the number of times we need to press that button to turn off the light above it. For example, we start by pressing each button in the second row one time to turn off all lights in the first row. We then need to press each button in the third row twice to turn off the buttons in the second row.

Notice that in this game, one-pass chasing works. After pressing the buttons in the last row all lights are off. When one-pass chasing turns off all lights we say the game is \textit{one-pass solvable}. Hence, the game shown in \Cref{fig:OnePassEx} is one-pass solvable. However, if we add a row  so that the board has six  rows, it is not hard to see that the game is no longer one-pass solvable because one-pass chasing no longer turns off all lights. 

\begin{figure}
    \centering
    \includegraphics[width=1.12in]{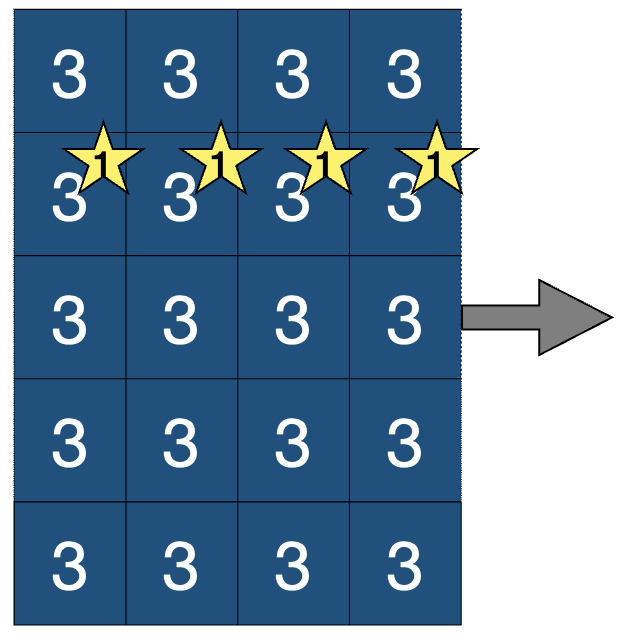}
    \includegraphics[width=1.12in]{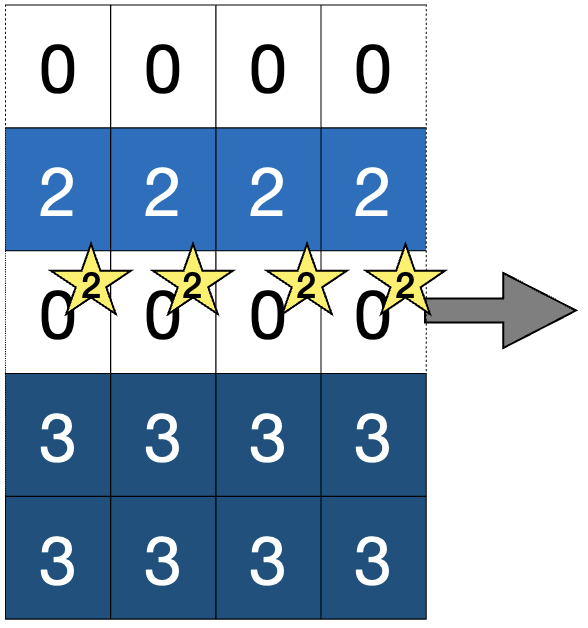}
    \includegraphics[width=1.12in]{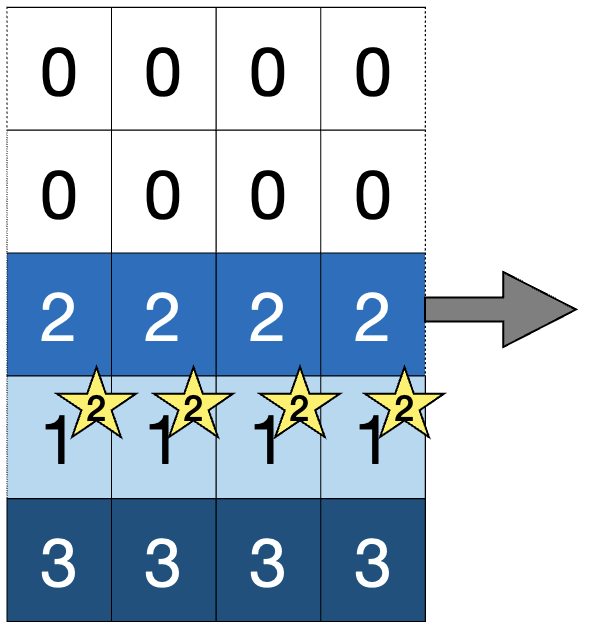}
    \includegraphics[width=1.12in]{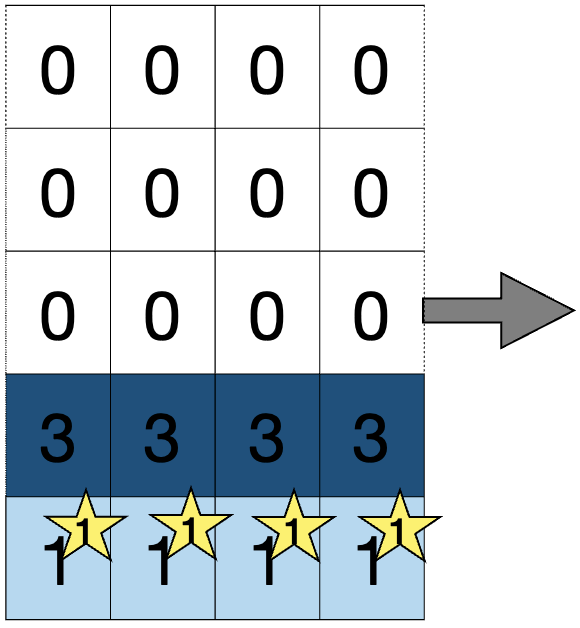}
    \includegraphics[width=1.12in]{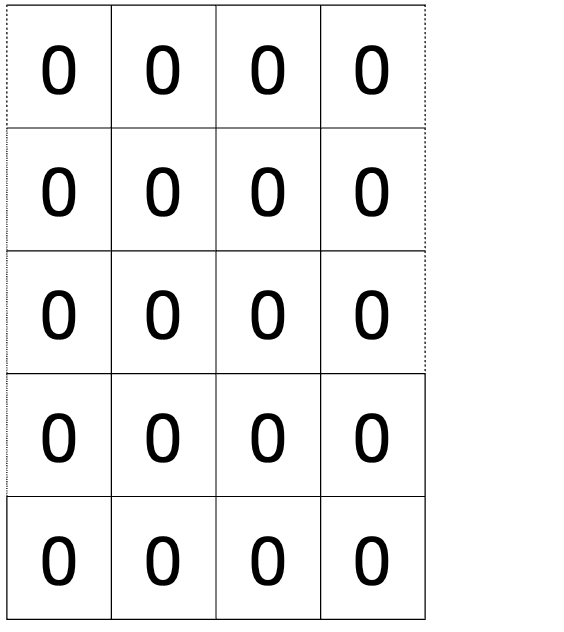}
    \caption{An example of one-pass chasing on a cylindrical Lights Out game in which the lights have four states and all lights start at the state of 3. The number in each star is the the number of times we press each button.}
    \label{fig:OnePassEx}
\end{figure}

\section{This Road Trip Game is Feeling Repetitive (A Recursive Model) }\label{sec:LightsBright}

In the remainder of this paper we analyze one-pass chasing on cylindrical Lights Out games in which the lights have $k$ states and all lights start at that same state, $k-q$, for $1 \le q \le k-1$. Since one-pass chasing is inherently recursive, we begin by building the recursion that gives the state of the lights in the $i^{th}$ row after turning off the lights in the $(i-1)^{st}$ row in the one-pass chasing process.

To build this recursion we need to keep track of the starting state of the lights in a row and  how many total times we  either press the button directly or we press a neighboring button that impacts the button; i,e., we need to keep track of  how many times we add one to the starting state during the one-pass  chasing process.  Since all lights start at the $k-q$ state and the grid is on a cylinder, the one-pass  chasing process adds the same amount to each light in a row.  Hence, so long as the board has at least three columns, whether or not a game is one-pass solvable is not affected by the number of columns of the board; it depends only on the number of rows.

We let $S_i$ be the state of the lights in the $i^{th}$ row after turning off the lights in the $(i-1)^{st}$ row and develop a formula for $S_i$ as follows. First, since the lights have $k$ states and begin at the $k-q$ state (or, equivalently, the $-q$ state), $S_i$ starts at $-q$. Then $S_i$ is affected when we press the buttons in the $(i-1)^{st}$ row to turn off the lights in the $(i-2)^{nd}$ row. We need to press the buttons in the $(i-1)^{st}$ row enough times so that the state of the lights in the $(i-2)^{nd}$ row becomes 0. We think of this number as $-S_{i-2}$. Similarly, $S_i$ is affected when we press the buttons in the $i^{th}$ row itself to turn off the lights in the $(i-1)^{st}$ row, and the number of times we need to press the buttons in the $i^{th}$ row is $-S_{i-1}$. Finally, $S_i$ is affected when we press the buttons on either side to turn off the lights above. The number of times we press these buttons is also $-S_{i-1}$. Thus, the following recursion gives the state of the lights in the $i^{th}$ row after turning off the lights in the $(i-1)^{st}$ row.

$$S_i =  -q-S_{i-2}-3S_{i-1},\ S_0=0,\ S_1=-q.$$

When the lights start at the $-1$ state (which we think of as the brightest state), the recursion becomes $S_i=-1-S_{i-2}-3S_{i-1}$ with $S_0=0$ and $S_1=-1$. The first 11 terms of this recursion are shown in \Cref{table:BiAllOn}.

\begin{table}[h]
\begin{tabular}{|c|c|c|c|c|c|c|c|c|c|c|c|}
\hline
$i$   & 0 & 1  & 2 & 3  & 4  & 5   & 6   & 7    & 8   & 9     & 10   \\ \hline
$S_i$ & 0 & -1 & 2 & -6 & 15 & -40 & 104 & -273 & 714 & -1870 & 4895 \\ \hline
\end{tabular}
\caption{The first 11 terms of the recursion $S_i=-1-S_{i-2}-3S_{i-1}$ for the game that begins with all lights set to the brightest state}
\label{table:BiAllOn}
\end{table}

Since the recursion $S_i$ gives the state of the lights in the $i^{th}$ row after pressing the buttons in the $i^{th}$ row in the one-pass process, a term like $S_4=15$ means that the lights in the fourth row start at $-1$ and are either pressed directly or are affected by pressing a neighboring button 16 times. The lights in the fourth row will be turned off if the lights reach the 0 state after pressing all lights in the row. Thus, if the lights have $k$ states, then the lights reach the 0 state if $S_4$ is divisible by $k$. This brings us to the following Key Fact.

\begin{KeyFact}\label{fact:KeyFact} A cylindrical game with $i$ rows of lights and in which the lights have $k$ states is one-pass solvable if and only if $S_i \equiv 0$ (mod $k$). 
\end{KeyFact}
For example, using \Cref{table:BiAllOn}, if the lights have five states and all lights start at the brightest state, then the game is one-pass solvable if the board has four, five, nine, or ten rows because $S_4$, $S_5$, $S_9$, and $S_{10}$ are divisible by five. On the other hand, if the lights start at the brightest state but only have two states, then the game is not one-pass solvable if the board has four or seven rows because $S_4$ and $S_7$ are both odd.

\section{Fibonacci Gets in the Car (Lights Outs Ties to the Fibonacci Sequence)}

Appealing to the Key Fact, our goal is now to determine when $S_i$ is congruent to 0 (mod $k$). To do this  we develop a relationship between the recursion $S_i$ and the Fibonacci sequence. The Fibonacci sequence has be extensively studied mod $k$ (see for example, \cite{renault}, \cite{ryder}, and \cite{wall}), and thus once we relate $S_i$ to the Fibonacci sequence, we will be able to use known results about the Fibonacci sequence to identify $i$'s for which $S_i \equiv 0$ (mod $k$). This will tell us how many rows a board should have in order for a game with $k$ state for the lights to be one-pass solvable.

Recall that the Fibonacci sequence is defined by $F_0=0$, $F_1=1$, and $F_i=F_{i-1}+F_{i-2}$. The first 13 terms of the Fibonacci sequence are shown in \Cref{table:Fib}. Comparing the terms $F_i$ in \Cref{table:Fib} with the  terms of $S_i$ with $q=1$  in \Cref{table:BiAllOn}, we see that $S_i = (-1)^iF_iF_{i+1}$ for $ 0 \le i\le 9$. More generally, \Cref{thm:RelatetoFib} shows that for any $1 \le q \le k-1$ and all $i$, $S_i=(-1)^iqF_iF_{i+1}$. The proof of \Cref{thm:RelatetoFib} requires the following identity.

\begin{Cassini} Let $F_i$ be the $i^{th}$ Fibonacci number. Then
    $$F_{i-1} F_{i+1} - F^2_{i} = (-1)^i.$$
\end{Cassini}


\begin{theorem}\label{thm:RelatetoFib} Let  $S_i = - q - S_{i-2} - 3S_{i-1}$ with $S_0=0$ and $S_1=-q$, and let $F_i$ be the \(i^{th}\) Fibonacci number.  Then $$S_i = (-1)^i q F_i F_{i+1}.$$

\end{theorem}

\begin{proof} We  use induction.  For the base case, we see that $S_0=0$ and $(-1)^0q F_0F_1= 1 \times q\times 0\times 1 =0$.

For the inductive step, let $n$ be an arbitrary natural number and assume that $S_i = (-1)^i q F_i F_{i+1}$   for all \( i \leq n \). We show that $S_{n+1} = (-1)^{n+1} q F_{n+1} F_{n+2}$.

By definition, \( S_{n+1} = -q - S_{n-1} - 3S_n \). By the inductive hypothesis, $S_{n-1} =(-1)^{n-1} q F_{n-1} F_n $ and $S_n = (-1)^n q F_{n} F_{n+1} $, and thus $$S_{n+1} = -q-(-1)^{n-1}q F_{k-1} F_n-3(-1)^n q F_{n} F_{n+1}.$$ Hence, we need to show
\begin{align}
    (-1)^{n+1} q F_{k+1} F_{k+2} &= - q - (-1)^{n-1}q F_{n-1} F_{n} - 3(-1)^n q F_{n} F_{n+1}.
\end{align}

Using the definition of $F_{n+1}$ and $F_{n+2}$ we see that
\begin{align}
    (-1)^{n+1}q F_{n+1} F_{n+2} &= (-1)^{n+1} q(F_{n} + F_{n-1})(F_{n+1} + F_{n}) \nonumber \\
    &= (-1)^{n+1}q (F_{n}F_{n+1} + F_{n}^2 + F_{n-1} F_{n+1} + F_{n-1} F_{n}).
\end{align}

Using Cassini's Identity, we replace the term \(F_{n-1} F_{n+1}\) in (2) with \(F_{n}^2 + (-1)^n\), so that Expression ($2$) becomes
\begin{align}
    (-1)^{n+1}q (F_{n}F_{n+1} + 2F_{n}^2 + (-1)^n + F_{n-1} F_{n}).
\end{align}

Since \(F_{n} = F_{n+1} - F_{n-1}\), it follows that $2F_n^2 = 2F_{n}(F_{n+1} - F_{n-1}) = 2F_{n} F_{n+1} - 2F_{n} F_{n-1}$. Thus,  Expression ($3$) becomes
\begin{align*}
    (-1)^{n+1}q (F_{n}F_{n+1} + 2F_{n} F_{n+1} - 2F_{n} F_{n-1} + (-1)^n + F_{n-1} F_{n}).
\end{align*}

Regrouping and distributing the $(-1)^{n+1}$ gives
\begin{align*}
    (-1)^{n+1} q ((-1)^n-F_nF_{n-1}+3F_nF_{n+1})&= q( (-1)^{2n+1}-(-1)^{n+1}F_nF_{n-1} +3(-1)^{n+1}F_nF_{n+1})\\
    &= -q-(-1)^{n-1}qF_{n-1}F_{n}-3q(-1)^{n}F_nF_{n+1} \\ 
\end{align*}as desired in Equation (1).\end{proof}

We can now determine board sizes for which the cylindrical game with $k$ states, where all buttons start in state $-q$, is one-pass solvable by studying when the product $F_iF_{i+1}$ is congruent 0 (mod $k$). For example, it is well-known that $F_i$ is even if $i \equiv 0$ (mod 3) and odd if $i \equiv 1$ or $i\equiv 2$ (mod 3). Thus, the game in which the lights have only two states (and hence only one possible starting state, $-1$) is one-pass solvable if and only if the number of rows of the board is congruent to 0 or 1 (mod 3).

The study of the Fibonacci sequence (mod $k$) began  with Wall's 1960 paper \textit{Fibonacci Series Modulo $m$} \cite{wall}. In particular, Wall showed that the Fibonacci sequence (mod $k$) is periodic (mod $k$) (Theorem 1, \cite{wall}) and studied when the terms of the Fibonacci sequence are congruent to 0 (mod $k$). Wall defined the \textit{restricted period} of the Fibonacci sequence (mod $k$), denoted $\alpha(k)$, to be the least positive integer such that $F_{\alpha(k)}\equiv 0$ (mod $k$), and proved that if we know $\alpha(k)$ then we can determine all $F_i$ such that $F_i \equiv 0$ (mod $k$).

\begin{theorem}[Theorem 3, \cite{wall}]\label{lem:Fib0} Let $F_i$ be the $i^{th}$ term of the Fibonacci sequence and let $\alpha(k)$ be the restricted period (mod $k$). Then $F_i \equiv 0$ (mod $k$) if and only if $i \equiv 0$ (mod $\alpha(k)$).
\end{theorem}

For example, \Cref{table:Fib} shows the first 13 terms of the original Fibonacci sequence along with the  sequence taken mod 3 and mod 4. Notice that $F_4 \equiv 0$ (mod 3). Thus,  $\alpha(3)=4$, and  by \Cref{lem:Fib0}, $F_i\equiv 0$ (mod 3) if and only if $i\equiv 0$ (mod 4). Taking the Fibonacci sequence mod 4, we see that the first 0 occurs at $F_6$. Hence,  $\alpha(4)=6$ and $F_i \equiv 0$ (mod 4) if and only if $i\equiv 0$ (mod 6). 

\begin{table}[h]
\begin{tabular}{|c|c|c|c|c|c|c|c|c|c|c|c|c|c|c|c|c|}
\hline
$i$           & 0 & 1 & 2 & 3 & 4 & 5 & 6 & 7  & 8  & 9  & 10 & 11 & 12 &13 & 14& 15  \\ \hline
$F_i$         & 0 & 1 & 1 & 2 & 3 & 5 & 8 & 13 & 21 & 34 & 55 & 89 & 144 & 233 & 377 & 610 \\ \hline
$F_i$ (mod 3) & 0 & 1 & 1 & 2 & \textbf{0} &2 & 2 & 1  & \textbf{0}  & 1  & 1  & 2  & \textbf{0} & 2 & 2 & 1  \\ \hline
$F_i$ (mod 4) & 0 & 1 & 1 & 2 & 3 & 1 & \textbf{0} & 1  & 1  & 2  & 3  & 1  & \textbf{0}  & 1 & 1 & 2  \\ \hline
\end{tabular}
\caption{The first 16 terms of the Fibonacci sequence $F_i$, the Fibonacci sequence (mod 3), and the Fibonacci sequence (mod 4). Notice that $\alpha (3)=4$ and $\alpha(4)=6$.}
\label{table:Fib}
\end{table}

 Since $S_i = (-1)^iqF_iF_{i+1}$, we can leverage the restricted period of the Fibonacci numbers to determine board sizes for which a game with $k$ states is one-pass solvable. 

\begin{theorem}\label{lem:nsatResPer}
 Let $\alpha(k)$ be the restricted period of the Fibonacci sequence (mod $k$). Then the cylindrical Lights Out game with $i$ rows, $k$ light states, and in which all lights start at the same state is one-pass solvable if $i \equiv 0$ (mod $\alpha(k)$) or $i \equiv -1$ (mod $\alpha(k)$).
\end{theorem}

\begin{proof} By the Key Fact, a  game with $k$ states and $i$ rows is one-pass solvable if $S_i \equiv 0$ (mod $k$). Since $S_i =(-1)^iqF_iF_{i+1}$, it follows that $S_i \equiv 0$ if either $F_i$ or $F_{i+1}$ is congruent to 0 (mod $k$). 
\end{proof}

It is tempting to read \Cref{lem:nsatResPer} as an if and only if statement and conclude that a game with  $k$ states and $i$ rows is only one-pass solvable when either $F_i$ or $F_{i+1}$ is congruent to 0 (mod $k$). Unfortunately, this is not always the case. For example, when $k=6$ we see that $F_2F_3 = 2\cdot 3 \equiv 0$ (mod 6) yet neither $F_2$ nor $F_3$ is congruent to 0 (mod 6). Similarly, if the lights have $k=6$ states and all start at the state of $q=3$, then $S_6=3\cdot F_6F_7=312\equiv 0$ (mod 6) but neither $F_6$ nor $F_7$ are congruent to 0 (mod 6). However, when $\mathbb{Z}_k$ is an integral domain (and thus has no zero divisors) then the reverse implication of \Cref{lem:nsatResPer} holds.

\begin{theorem}\label{cor:noZeroDiv} Let $\alpha(k)$ be the restricted period of the Fibonacci sequence (mod $k$). If $\mathbb{Z}_k$ is an integral domain, then the cylindrical Lights Out game with $i$ rows, $k$ lights states, and in which all lights start at the same state is one-pass solvable if and only if $i \equiv 0$ (mod $\alpha(k)$) or $i \equiv -1$ (mod $\alpha(k)$).
\end{theorem}

 For example, $\mathbb{Z}_3$ is an integral domain and from \Cref{table:Fib} we see that $\alpha(3)=4$. Therefore, a game with $i$ rows, 3 states for the lights, and in which all lights start at the same state is one-pass solvable  if and only if $i \equiv 0$ (mod 4) or $i\equiv 3$ (mod 4).

We end with a discussion of computing $\alpha(k)$, which unfortunately, does not come equipped with a nice formula. When $k$ is small we can directly compute $\alpha(k)$ by writing out the Fibonacci sequence (mod $k$) and noting the location of the first 0. For example, from \Cref{table:Fib} we see that $F_5$ is the first term of the sequence divisible by 5.  Hence, $\alpha(5)=5$. If $\alpha(k)$ is too large to compute directly, there are a few results that allow us to compute $\alpha(k)$ if we know $\alpha(p)$ for all prime $p$ that divide $k$. In particular, we rely on the following theorems from \cite{renault} and \cite{wall}.

\begin{theorem}[Theorem 1, \cite{renault}]\label{thm:lcm}
    Let $\text{lcm}$ denote the least common multiple and let $\alpha(k)$ be the restricted period of the Fibonacci sequence (mod $k$). Then $$\alpha(\text{lcm}(k_1,k_2)) = \text{lcm}(\alpha(k_1),\alpha(k_2)).$$
\end{theorem}

\begin{theorem}[Theorem 5, \cite{wall}]\label{thm:alphaprime}
    If $p$ is  an  odd prime and $\alpha(p^2) \ne \alpha(p)$ then $$\alpha(p^s)=p^{s-1}\alpha(p).$$
\end{theorem}

\begin{theorem}[Theorem 2, \cite{renault}]\label{thm:alph2}
   If $p=2$ and $s \ge 3$,  then $\alpha(2^s)=2^{s-3}\alpha(4) = 2^{s-3}\cdot 6$. 
\end{theorem}

For example, using \Cref{thm:lcm}, and $\alpha(3)$ and $\alpha(4)$ from \Cref{table:Fib},  $$\alpha(12) = \alpha(\text{lcm}(3,4)) = \text{lcm}(\alpha(3),\alpha(4)) = \text{lcm}(4,6)=12.$$
We now have an infinite class of board sizes for which the game with 12 states  is one-pass solvable. Specifically, if the number of rows of the board is congruent to 0 or 11 (mod 12) then the game with 12 states for the lights and in which all lights start at the same state will be one-pass solvable. (However, since $\mathbb{Z}_{12}$ is not an integral domain, we have not necessarily found all possible board sizes for which the game is one-pass solvable.)

\Cref{thm:alphaprime} is particularly powerful when $k$ is a power of 5. Since $F_5=5$, it follows that $\alpha(5)=5$ and $\alpha(25) \ne \alpha(5)$. Thus, $p=5$ satisfies the premise of \Cref{thm:alphaprime}, and so $\alpha(5^s)=5^{s-1}\alpha(5)=5^s$.

\begin{theorem}
   Let $s$ be a natural number. The  cylindrical Lights Out game with $i$ rows in which the lights have $5^s$ states and all lights start at the same state is one-pass solvable if $i \equiv 0$ (mod $5^s$) or $i \equiv -1$ (mod $5^s$). 
\end{theorem}

Combining \Cref{thm:lcm},  \Cref{thm:alphaprime}, and \Cref{thm:alph2}, we see that if we write $k$ in terms of its prime factorization $k=\prod_{j=1}^rp_j^{s_j}$ where $p_j$ is prime and $s_j$ is a natural number then $$\alpha(k) = \text{lcm}(\alpha(p_1^{s_1}),\alpha(p_2^{s_2}),\dots,\alpha(p_r^{s_r})).$$ Thus, we can calculate $\alpha(k)$ if we know $\alpha(p_i)$ (and that $\alpha(p_i) \ne \alpha(p_i^2)$) for all $i$. This allows us to determine an infinite class of board sizes for which the cylindrical Lights Out game with $k$ lights states and in which all lights start at the same state is one-pass solvable.

We end with a final example. Consider a game in which the lights have 1200 states and all lights start at the same state.  To determine a collection of board sizes for which this game is one-pass solvable we calculate $\alpha(k)$. Since $k=1200=5^2\cdot 2^4\cdot 3$, it follows that $$\alpha (1200)=\text{lcm}(\alpha(5^2),\alpha(2^4),\alpha(3)) = \text{lcm}(25,12,4) = 300.$$ Therefore, this game  will be one-pass solvable if the number of rows is congruent to 0 or $-1$ (mod 300). Due to the size of the board, we do not recommend this game for a long car ride.


\begin{thebibliography}{10}
\bibitem{anderson} Anderson, M., and Feil, T. (1998). Turning Lights Out with Linear Algebra. Mathematics Magazine, 71(4), 300–303. 
\bibitem{ferman} Ferman, L., and Arangala, C. (2024). A family of multicolor lights out games. Discrete Mathematics, Algorithms and Applications, 16(03), 2350023.
\bibitem{leach} Leach, C. D. (2017). Chasing the Lights in Lights Out. Mathematics Magazine, 90(2), 126–133. 
\bibitem{martin} Martín-Sánchez, Ó., and Pareja-Flores, C. (2001). Two Reflected Analyses of Lights Out. Mathematics Magazine, 74(4), 295–304. 
\bibitem{miss} Missigman, J., and Weida, R. (2001). An easy solution to mini lights out. Mathematics Magazine, 74(1), 57-59.

\bibitem{renault}  Renault, M. (2013). The Period, Rank, and Order of the $(a,b)$-Fibonacci Sequence Mod $m$, Mathematics Magazine, 86,  372-380.

\bibitem{ryder} Ryder, J. (1996). Exploring Fibonacci Numbers Mod M, The College Mathematics Journal, 27(2), 122-124.

\bibitem{torrence} Torrence, B. (2011). The Easiest Lights Out Games. The College Mathematics Journal, 42(5), 361–372. 


\bibitem{wall} Wall, D. (1960). Fibonacci Series Modulo m. The American Mathematical Monthly \textit{Am. Math. Mon}, 67(6), (1960), 525-532.

\end{thebibliography}
\end{document}